\documentclass[11pt]{amsart}
\usepackage{amsmath,amssymb,txfonts}
\usepackage{amssymb}
\usepackage{amsxtra}
\usepackage{amsmath}
\usepackage{txfonts}

\usepackage{hyperref}

\usepackage[cp1252]{inputenc}

\usepackage{mathrsfs}
\usepackage{graphicx}

\usepackage[active]{srcltx}

\usepackage[notref,notcite]{}

\newtheorem{conjecture}{Conjecture}[section]
  \newtheorem{theorem}{Theorem}[section]
  \newtheorem{lemma}{Lemma}[section]
  \newtheorem{proposition}{Proposition}[section]

  \newtheorem{definition}{Definition}[section]
  \newtheorem{remark}{Remark}[section]

\numberwithin{equation}{section}

   \newcommand{\beqn}{\begin{eqnarray}}
   \newcommand{\eeqn}{\end{eqnarray}}
   \newcommand{\beqs}{\begin{eqnarray*}}
   \newcommand{\eeqs}{\end{eqnarray*}}
   \newcommand{\ban}{\begin{eqnarray*}}
   \newcommand{\nan}{\end{eqnarray*}}
   \newcommand{\beq}{\begin{equation}}
   \newcommand{\eeq}{\end{equation}}

\renewcommand{\det}{\mbox{det}}

\begin{document}
\allowdisplaybreaks
\arraycolsep=1pt

\title{Modified Futaki invariant and equivariant Riemann-Roch formula}
\author{Feng Wang\ \ \ \ \ \ Bin  $\text{Zhou}^*$\ \ \ \ \ \ Xiaohua  $\text{Zhu}^{**}$}


\address{Feng Wang, Department of Mathematics, Zhejiang University, Hangzhou 310027, China}

\address{Bin Zhou, School of Mathematical Sciences, Peking
University, Beijing 100871, China; and
Mathematical Sciences Institute, The Australian National University, Canberra, ACT 2601, Australia.}

\address{Xiaohua Zhu, School of Mathematical Sciences, Peking
University, Beijing 100871, China.}

\email{wf19870517@163.com\ \ \ bzhou@pku.edu.cn\ \ \ xhzhu@math.pku.edu.cn}

\thanks {*Partially supported by NSFC 11101004 and ARC DECRA}

\thanks {** Partially supported by the NSFC Grants 11271022 and 11331001.}

\begin{abstract}
In this paper, we  give a new version of the modified Futaki invariant for a test configuration associated to  the soliton action on a Fano manifold.  Our version will naturally come from toric test configurations defined by Donaldson for toric manifolds.  As an application, we show that the modified $K$-energy is proper for  toric invariant K\"ahler potentials on a   toric Fano manifold. 
\end{abstract}

 \maketitle


\section{Introduction}

\vskip 10pt

Let $(M,g)$ be a Fano manifold with a K\"ahler form  $\omega_g\in 2\pi c_1(M)$ of $g$.  Denote   $\eta(M)$  to be the linear space of  holomorphic vector fields on $M$. Then by Hodge Theorem,  for  any   $X\in \eta(M)$,   there exists  a unique smooth  complex-valued  function
$\theta_X(g)$ of $M$ such that
\beqs
\begin{cases}
i_X\omega_g=\sqrt{-1}\bar\partial\theta_X(g), &\\[5pt]
\int_Me^{\theta_X(g)}\omega_g^n=\int_M\omega_g^n.&
\end{cases}
\eeqs

In \cite {TZ2}, Tian and Zhu introduced the modified Futaki invariant on $\eta(M)\times \eta(M)$,
\beq\label{tian-zhu-inv}
 F_{X}(v)=\int_M v(h_g-\theta_X(g))e^{\theta_X(g)}\frac{\omega_g^n}{n!}, \ \ X, v\in \eta(M),
\eeq
where $h_g$ is  the Ricci potential of $g$  such that
$${\rm Ric}(\omega_g)-\omega_g=\sqrt{-1}\partial\bar\partial h_g.$$
It was shown  there that $F_X(v)$ is a holomorphic invariant independent of the choice of $g$ with $\omega_g\in 2\pi c_1(M)$,  and so  it  defines an obstruction to the existence of K\"ahler-Ricci solitons  with respect to an element $X\in \eta_r(M)$, where $\eta_r(M)$ is the reductive part of $\eta(M)$.   In particular, when $X=0$,   $F_{X}(v)$ is classical Futaki invariant \cite{Fut}. It was also proved by Tian and Zhu that there exists a unique $X$ such that $F_{X}(v)=0$, $\ \forall \ v\in \eta_r(M)$.
For convenience, we call such  $X$ the {\it soliton vector field} on $M$.

Recently, by using Ding-Tian's  idea of generalizing Futaki invariant \cite {DT}, Xiong and Berman, gave a generalization of the modified Futaki invariant $F_{X}(\cdot)$ for any special degeneration associated to  $X$, independently \cite{Xi, Be2}.   As a consequence, they both proved that  $F_{X}(\cdot)$ is nonnegative if $M$ admits a K\"ahler-Ricci soliton. Berman also gaves an algebraic formula for  $F_{X}(\cdot)$, which depends on weights of the automorphisms group on holomorphic sections of multi-line bundles on the center fibre induced by the test configuration.   The purpose of present paper is  to  define  the modified Futaki invariant   $F_{X}(\cdot)$   for   any  test configuration associated to  $X$.
Our motivation is inspired by Berman's  algebraic formula for   special degenerations and is to  modify his formula for  general  test configurations. Then by applying  the equivariant Riemann-Roch formula with a $2$-dimensional torus action we show that our definition coincides with  Xiong-Berman's for  special degenerations.  Our definition also includes Tian and  Donaldson's   generalized  Futaki invariant as a special case \cite {T1, D1} when $X=0$.

 As examples, we compute  the new version of the modified Futaki invariant  for any toric degeneration on  toric manifolds  introduced  by Donaldson \cite{D1}. Then  by using the method in \cite{ZZ1}, we are able to prove

\begin{theorem}\label{WZZ}
Any toric Fano manifold  is modified $K$-stable  for any toric  degeneration. Furthermore, the  modified $K$-energy is proper   for  toric invariant K\"ahler potentials.

\end{theorem}

Theorem \ref{WZZ}  gives a new proof of Wang-Zhu's result for the existence of    K\"ahler-Ricci solitons on any toric Fano manifold \cite{WZ}.   We can also  study   the existence of conical   K\"ahler-Ricci solitons on toric Fano manifolds by  showing  the properness of modified   Log $K$-energy.  As a  consequence, we give a  new proof of  Datar-Guo-Song-Wang Theorem in \cite{DGSW}. In particular, we have

\begin{theorem}\label{conical-theorem}  Let $ X$ be a soliton vector field on a toric Fano manifold $M$.\footnote{$X$ is not unique  in general for the existence of  conical K\"ahler-Ricci solitons while it is always  unique modulo ${\rm Aut}^0(M)$ for the existence of  K\"ahler-Ricci solitons \cite{TZ2}.}   Then for any $\beta\le 1$ there exists  a unique toric invariant conical K\"ahler-Ricci soliton which has conical angle  $2\pi\beta$  along  each face divisor $D_i$ of $M$.

\end{theorem}

 We note that the above energy argument was  used by other people, such as in \cite{JMR, LS, T3, LZ} to study the  conical   K\"ahler-Einstein metrics on general  Fano manifolds.  Theorem \ref {WZZ} and Theorem \ref{conical-theorem} will be proved in Section 2-3 and Section 4, respectively.

\vskip10pt
\noindent {\bf Acknowledgements.}   The  third author would like to thank
professor Gang  Tian for  his interest  to the paper  and sharing his insight  in K\"ahler geometry. 

\vskip 20pt

\section{New version  of  modified Futaki invariant}

\vskip 10pt

According to \cite{D1},  a  {\it test-configuration} on a  Fano manifold $M$ is a scheme $\mathcal M$ with a $\mathbb C^*$- action which consists of two integredients:
\begin{enumerate}
\item an flat $\mathbb C^*$-equivarant map $\pi: \mathcal M\to \mathbb C$ such that $\pi^{-1}(t)$ is biholomorphic to $M$ for any $t\neq 0$;

\item an holomorphic line bundle $\mathfrak  L$ on $\mathcal M$ such that $\mathfrak L|_{\pi^{-1}(t)}$ is isomorphic to $K_M^{-r}$ for some integer $r>0$  for any $t\neq 0$.
\end{enumerate}

\begin{definition}
$\mathcal M$ is called a test-configuration associated to the soliton action  induced by $X$ if $\sigma_t^v$ communicate  to $\sigma_t^X$,  where  $\sigma_t^X$ and $\sigma_t^v$  are  two  lifting one-parameter subgroups on
$\mathcal M$  induced by $X$ and the  holomorphic vector field $v$ associated to  the $\mathbb C^*$- action, respectively. If furthermore  the center fibre  $M_0=\pi^{-1}(0)$ is a normal variety we call  $\mathcal M$ is a special degeneration.
In particular, if $\mathcal M\cong M\times  \mathbb C$, $\mathcal M$ is called a trivial  test-configuration.
\end{definition}

For simplicity, we let $L=\mathfrak L|_{ M_0}$.  Let $\sigma_t^X(k), \sigma_t^v(k)$  be two induced one-parameter subgroups on  $H^0( M_0, L^k )$ by  $\sigma_t^X,  \sigma_t^v$, respectively. Denote by $\{e^{X_\alpha^k}\}$ and
 $\{e^{v_\alpha^k}\}$  be  eigenvalues of  actions $\sigma_1^X$ and $\sigma_1^v$. We set
$$ S_1=\sum_i e^{\frac{X_\alpha^k}{k}}v_\alpha^k, \ \ S_2=\frac{1}{2}\sum_i e^{\frac{X_\alpha^k}{k}}\frac{X_\alpha^k}{k}\frac{v_\alpha^k}{k}.$$
Then
$$S_1=\frac{\partial}{\partial t}{\rm trace}(e^{sX^k+tv^k})|_{s=\frac{1}{k},\ t=0}, \ \ S_2=\frac{1}{2 k^2}\frac{\partial}{\partial s}\frac{\partial}{\partial t}{\rm trace}(e^{sX^k+tv^k})|_{s=\frac{1}{k},t=0},$$
where $X^k= (X_\alpha^k)$, $v^k= (v_\alpha^k)$ are two vectors as  elements of  Lie algebra associated to    $\sigma_t^X(k)$, $\sigma_t^v(k)$, respectively.
Our  observation is  that both $S_1$ and $S_2$ can be computed by  the equivariant Riemann-Roch formula with $G=(S^1)^2$-action. In fact
\beqn\label{equiva-Riemann-Roch formula}  {\rm trace}(e^{sX^k+tv^k})=\int_{M_0} {\rm ch}^G(-k L){\rm Td}^G(X_0),
\eeqn
where ${\rm ch}^G(-kL)$ is a $G$-equivariant Chern character of  multi-line bundle  $-kL$ and ${\rm Td}^G(M_0)$ is a  $G$-equivariant Todd character of $M_0$ \cite{AS}.   In particular, for a  special degeneration   associated to the soliton action, we can compute both $S_1$ and  $S_2$ precisely in the following.

According to \cite{DT}, for  a special degeneration,  there exists a hermitian metric  $h$ on $(X_0, L)$ such that  curvature $c(h, L)$ is a  r-multiple of  an admissible metric $g$ with  property:
there exists a $L^p$-integrable function $h_g$ (for any $p \ge 0$)  with respect to $g$ such that

i)  ${\rm Ric}(\omega_g)-\omega_g=\sqrt{-1}\partial\bar\partial h_g$, on the smooth part of $M_0$;

ii) $v(h_g)$ is $L^1$-integrable with respect to $g$.
We note that $v$ is an admissible  holomorphic vector field  on $M_0$  \cite{DT}.

For any admissible   holomorphic vector field $w$ on $M_0$,  we define a function  by
\beqn \label{potential-vector}
\theta_{w}=-\frac{L_w h}{h}.
\eeqn
Then a direct computation shows
\beqs
\sqrt{-1}\overline\partial \theta_{w}=i_w \omega_g,
\eeqs
and consequently
\begin{align}\label{equ}
\Delta\theta_w=\frac{L_w\omega_g^n}{\omega_g^n}.
\end{align}

\begin{lemma}
 $ \theta_{w}$  satisfies
 \beqn \label{potential-vector-normalization}
\triangle \theta_{w} +w(h_g)+\theta_{w}=0.
\eeqn
\end{lemma}

\begin{proof}  It suffices  to verify (\ref{potential-vector-normalization}) on the smooth part of $X_0$.
Since
 $${\rm Ric}(\omega_g)-\omega_g={\rm Ric}(\omega_g)-{\rm Ric}(h)=\sqrt{-1}\partial\bar{\partial}\log\frac{h}{\omega_g^n},$$
we have
  $$h_g=\log\frac{h}{\omega_g^n}+const.$$
  It follows
  $$w(h_g)=w(\log\frac{h}{\omega_g^n}).$$
Thus
\begin{align}
\triangle \theta_{w} +w(h_g)+\theta_{w}=\frac{L_w\omega_g^n}{\omega_g^n}+w(\log\frac{h}{\omega_g^n})-\frac{L_w h}{h}=0.\notag
\end{align}
\end{proof}

\begin{lemma}
\label{S_1+S_2} Let $ \theta_{X}$ and  $\theta_{v}$ be defined by (\ref{potential-vector}) for the vectors $X$ and $v$, respectively.
 Then,  instead of  $k$  by $kr$,  we have
\beqn \label{S_1}  S_1&=&k^{n+1}\int_{\mathcal{M}_0} \theta_{v} e^{\theta_{X}}\frac{\omega_g^n}{n!} \nonumber\\[5pt]
 &&+\frac{k^n}{2}
\left[ n\int_{\mathcal{M}_0} \theta_{v} e^{\theta_{X}}\frac{\omega_g^n}{n!}+ \int_{\mathcal{M}_0} \theta_{X}\theta_{v} e^{\theta_{X}}\frac{\omega_g^n}{n!}
  -\int_{\mathcal{M}_0} v(h_g-\theta_{X}) e^{\theta_X}\frac{\omega_g^n}{n!}\right]+O(k^{n-1}),\\[7pt]
 \label{S_2} S_2
&=& \frac{k^{n}}{2}\int_{\mathcal{M}_0}  \theta_{v} \theta_X e^{\theta_{X}}\frac{\omega_g^n}{n!}
+O(k^{n-1}) .
\eeqn
\end{lemma}

\begin{proof}
Since
$${\rm ch}^G(-kK_{M_0})=e^{k\omega_g+ks\theta_{X}+kt\theta_{v}}$$
 and
$${\rm Td}^G(M_0)=1+\frac{1}{2}c^G_1+ \sum_{i+j\ge 2} a_{ij}s^it^j+2l-{\rm forms} ~(l\ge 2),$$
where $c^G_1$ is the   first Chern $G$-equivariant form, we have
\beqn
&&\frac{d}{dt} \Big |_{t=0}{\rm trace}(e^{\frac{1}{k}X+tv}) \nonumber\\[4pt]
&&=k\int_{M_0}\theta_{v} e^{k\omega_g+\theta_{X}}{\rm Td}^G(M_0)+\int_{M_0} e^{k\omega_g+\theta_X}\frac{d}{dt}{\rm Td}^G(M_0)  \nonumber\\[5pt]
&&=k^{n+1}\int_{M_0} \theta_{v} e^{\theta_X}\frac{\omega_g^n}{n!}\wedge(1+ \frac{1}{2}c^G_1)+\frac{1}{2}k^n\int_{M_0} e^{\theta_{X}}\frac{\omega_g^{n-1}}{n!}\wedge c^G_1\nonumber\\
&&+\int_{M_0} e^{k\omega_g+\theta_X}(\frac{d}{dt}c^G_1)\frac{\omega_g^n}{n!}+O(k^{n-2}). \nonumber
\eeqn
Note that
$$c^G_1={\rm Ric}(\omega_g)-s\Delta \theta_{V_0}-t\Delta \theta_{W_0}.$$
Thus
\beqn
\label{estimate-1} &&\frac{d}{dt} \Big |_{t=0}{\rm trace}(e^{\frac{1}{k}X+tv}) \nonumber\\
&&=k^{n+1}\int_{M_0} e^{\theta_{X}}\theta_{v}\frac{\omega_g^n}{n!}+\frac{k^n}{2} \left(\int_{M_0} e^{\theta_{X}}\theta_{v} S\frac{\omega_g^n}{n!} -\int_{M_0}e^{\theta_{X}}\theta_{v}\triangle \theta_{X} \frac{\omega_g^n}{n!}- \int_{M_0} e^{\theta_{X}} \triangle\theta_{v} \frac{\omega_g^n}{n!}\right)+O(k^{n-1}).
\eeqn
On the other hand, using the integration by parts, it is  easy to see that
\beqn\label{estimate-2}
\int_{M_0} \theta_{v}e^{\theta_{X}} S\frac{\omega_g^n}{n!}&=&n\int_{M_0} \theta_{v}e^{\theta_{X}}\frac{\omega_g^n}{n!}+\int_{M_0} \theta_{v} e^{\theta_{X}}\triangle h_g \frac{\omega_g^n}{n!}\nonumber\\
&=&n\int_{M_0} \theta_{v} e^{\theta_{X}}\frac{\omega_g^n}{n!}-\int_{M_0} \theta_{v} e^{\theta_{X}} X(h_g)\frac{\omega_g^n}{n!}-\int_{M_0}  v(h_g)e^{\theta_X}\frac{\omega_g^n}{n!}.
\eeqn
Hence, by using the relation (\ref{potential-vector-normalization}) for $\theta_X$,  we will get (\ref{S_1}) from (\ref{estimate-1}) and  (\ref{estimate-2}) immediately.

The proof of  (\ref{S_2}) is easy. We skip it.
\end{proof}

Let $N_k={\rm dim}H^0(M_0, L^k)$.  Then by the Riemann-Roch formula, we have
\beq
N_k=\int_{M_0} ch(-kK_{M_0}) Td({M_0})=\int_{M_0} e^{k\omega_g}Td(M_0).
\eeq
Note that the Todd class is given by
$$Td(M_0)=1+\frac{1}{2}c_1(M_0)+\cdot\cdot\cdot,$$
where $c_1(M_0)$ is the first  Chern class of $M_0$. Thus
\beq
N_k= A_0 k^n + B_0k^{n-1}+O(k^{n-2}),
\eeq
where
\beq
A_0=\int_{M_0} \frac{\omega_g^n}{n!}={\rm Vol}(M_0), \ \ B_0=\frac{1}{2}\int_{M_0} S\,\frac{\omega_g^n}{n!}=\frac{n}{2}{\rm Vol}(M_0).
\eeq
Here $S$ is the scalar curvature of $\omega_g$.

By Lemma \ref{S_1+S_2} , we  can write $\frac{S_1-S_2}{kN_k}$ as an expansion in $k^{-1}$,
\beqn \label{expansion}
-\frac{S_1-S_2}{kN_k}= F_0+F_1k^{-1}+o(k^{-2}).
\eeqn

\begin{proposition}\label{definition-f_1}
For a special degeneration on a Fano manifold, we have
$$F_1=\frac{1}{2}F_0=
\frac{1}{2{\rm Vol}(M_0)}\int_{M_0} v(h_g-\theta_{X})  e^{\theta_X}\,\frac{\omega_g^n}{n!}.
$$
\end{proposition}

\begin{proof}
By using the relation (\ref{potential-vector-normalization}) for $\theta_v$, we have
$$\int_{M_0}  v(h_g-\theta_{X}) e^{\theta_X}\,\frac{\omega_g^n}{n!}=-\int_{M_0}  \theta_v e^{\theta_{X}}\,\frac{\omega_g^n}{n!}.$$
Then the Proposition follows from Lemma \ref{S_1+S_2} immediately.
\end{proof}

For a  general test-configuration  associated to the soliton action,  by  the equivariant Riemann-Roch formula  (\ref{equiva-Riemann-Roch formula}),
 we write $S_1$ and $S_2$ formally as,
\beqs
&&S_1=Ak^{n+1}+Bk^n +O(k^{n-1}),\\
&&S_2=Ck^{n}+Dk^{n-1} +O(k^{n-2}).
\eeqs
Then the invariant $F_1$ in (\ref{expansion}) is equal to $\frac{2(B-C)-nA}{2A_0}$. 
We call this quantity the {\it modified Futaki invariant} for  the test-configuration  
$\mathcal M$  associated  to the soliton vector field $X$.

\begin{remark}
1) In \cite{Be2},  Berman defines  the modified Futaki invariant by $F_0$ for any special degeneration  associated  to the soliton vector field  $X$.  Proposition \ref{definition-f_1} means  that our  definition coincides  Berman's case.
But in general $F_0$ will be  different to $F_1$ as showed in next section for  toric degenerations  on a toric Fano manifold.  In fact,   we  will show  that  the invariant $F_1$  comes natrually from  the study of  
 modified K-energy on toric manifolds (cf. Section 3).

2) Proposition \ref{definition-f_1}  also shows that  the Donaldson invariant $F_1$ in \cite {D1} coincides with  the generalized Futaki invariant defined by Tian in \cite{T1} for special degenerations.
\end{remark}

As in \cite {T1, D1},  we introduce a notation of  modified $K$-stability for any  Fano manifold $M$  via the quantity $F_1$.

\begin{definition}
A Fano manifold $M$ is called modified $K$-semi-stable  if $F_1\ge 0$ for any  test-configuration   associated to the soliton action of $M$ and  $M$ is modified $K$-stable  if in addition $F_1=0$ happens if and only if
the test-configuration is trivial.
\end{definition}

Due to the celebrated  solving of Yau-Tian-Donaldson's conjecture for the existence of K\"ahler-Einstein metrics \cite{T3, CDS}, we propose the following  generalized  Yau-Tian-Donaldson's conjecture for  the existence of K\"ahler-Ricci solitons.

\begin{conjecture}\label{generalized-conjecture}
 A Fano manifold $M$  admits a  K\"ahler-Ricci soliton  if and only if  $M$ is modified $K$-stable.
\end{conjecture}

In the remaining  sections, we verify  Conjecture \ref{generalized-conjecture}  in case of toric Fano manifolds (also for general conical  K\"ahler-Ricci solitons).

\vskip 20pt

\section{Modified Futaki invariant for toric degenerations}

\vskip 10pt

In this section we compute the modified Futaki invariant $F_1$ for a toric degeneration on a toric Fano manifold $M$.
Let $T=T^n_\mathbb{C}=(\mathbb{C}^*)^n=(S^1)^n\times \mathbb{R}^n$ be torus action on $M$  and denote
$G_0=(S^1)^n$. Choose an $G_0$-invariant K\"ahler metric $g$ with $\omega_g\in 2\pi c_1(M)$.
The open dense orbit of $T_{\mathbb C}^n$ in $M$
induces an global coordinates $(w_1,...,w_n)\in (\mathbb C^*)^n$.
To do the reduction we use the affine logarithmic coordinates
$z_i=\log w_i=\xi_i+\sqrt{-1}\eta_i$.
Then $\omega_g$ is determined by a strictly convex function $\varphi_0$
which depends only on $\xi_1, ... ,\xi_n\in{\mathbb R^n}$
in the coordinates $(z_1, ......, z_n)$, namely
$\omega_g =\sqrt{-1}\partial\bar{\partial}\varphi_0$ on $(\mathbb C^*)^n$.
Since the torus action $T$ is Hamiltonian, there exists a moment map
$m: M\rightarrow \mathfrak t^*$,
where $\mathfrak t^*$ is the dual of the Lie algebra of $T$ which can be identified with $\mathbb R^n$.
By the convexity theorem
the image is a convex polytope in $\mathbb R^n$. Moreover,
the moment map can be given by
$$(m_1, ......, m_n)=\nabla \varphi_0=\left(\frac{\partial \varphi_0}{\partial \xi_1}, ...... , \frac{\partial \varphi_0}{\partial \xi_n}\right).$$
Denote the image by $P=D\varphi_0(\mathbb R^n)$.
Then $P$ is a convex polytope represented by a set of inequalities of the form
(up to translation of coordinates)
\beq\label{polytope}
P=\{x\in \mathbb R^n:\ \langle x, \ell_i\rangle \leq 1, \ i=1, 2, \cdots, d\},
\eeq
where $\ell_i$ is the  outer normal vector  to a face of $P$ and $d$ is the number of faces of $P$.
This polytope is independent of the choice of the metric $g$
in $2\pi c_1(M)$. See \cite{Ab1, Ab2, Gu} for more details.

On the other hand, the soliton vector field $X$ can be written as
$X=\sum_{i=1}^n \theta_iw_i\frac{\partial}{\partial w_{i}}=\sum_{i=1}^n \theta_i\frac{\partial}{\partial z_{i}}$. Let $\theta_X(\omega_g)$ be the potential function determined by
$$i_X\omega_g=\sqrt{-1}\bar\partial\theta_X(\omega_g),$$
then $\theta_X(\omega_g)=X(\varphi_0)+c$ for some $c\in\mathbb R^n$.
By (\ref{potential-vector-normalization}),  it is easy to see
\beq
\int_M\theta_X(\omega_g)e^{h_g}\frac{\omega_g^n}{n!} =0.
\eeq
Note that
$$\int_M\theta_{\frac{\partial}{\partial z_i}}(\omega_g)e^{h_g}\frac{\omega_g^n}{n!} =C\int_{\mathbb R^n}\frac{\partial\varphi}{\partial \xi_i}e^{-\varphi},\  i=1,2, \cdot\cdot\cdot,$$
for a constant $C$. Then $c=0$ under the above normalization, and we have
\beq
\theta_{X}(\omega_g)= \sum_{i=1}^n \theta_i\frac{\partial \varphi_0}{\partial z_{i}}
=\frac{1}{2}\sum_{i=1}^n \theta_i\frac{\partial \varphi_0}{\partial \xi_{i}}=
\sum_{i=1}^n{\theta_{i}x_{i}}:=\theta(x)
\eeq
in the symplectic coordinates. One can see that $\theta(x)$ is also independent of the choice of metric $g$.

Accordimg to  \cite{D1}, a toric degeneration is 
induced by positive rational, piecewise linear functions on  $P$.  
Note that a {\it piecewise linear(PL)} convex function  $u$ on $P$ is of the form
$$u= \text{max}\{u^1, ..., u^r\},$$
 where $u^\lambda=\sum a_i^\lambda x_i + c^\lambda, \lambda=1,...,r,$
for some vectors $(a_1^\lambda, ..., a^\lambda_n)\in \mathbb R^n$ 
and some numbers
$c^\lambda\in \mathbb R$. $u$ is called a {\it rational piecewise linear} convex function if the
coefficients $a^\lambda_i$ and numbers $c^\lambda$ are all rational.

For a  positive rational PL convex function $u$ on $P$,  
we choose a positive integer $R$ so that
$$Q=\{(x,t)\ |\  x\in P,\  0<t<R-u(x)\}$$
is a convex polytope in $\mathbb R^{n+1}$. Without loss of
generality,  we may assume that the coefficients $a^\lambda_i$ 
are integers and  $Q$ is an integral polytope.
Otherwise we replace $u$ by $lu$ and $Q$ by $lQ$ for some integer
$l$, respectively. Then the $n+1$-dimensional polytope $Q$ determines an
$(n+1)$-dimensional toric variety $M_Q$ 
with a holomorphic  line bundle $\mathfrak L\rightarrow M_Q$.  
Note that the face $\bar Q\cap \{\mathbb R^n\times\{0\}\}$
of $Q$ is a copy of the $n$-dimensional polytope $P$,  so we have a 
natural embedding $i: M\rightarrow M_Q$ such that $\mathfrak L|_M=-K_M$. Decomposing 
the torus action $T_{\mathbb C}^{n+1}$ on $M_Q$ as 
$T_{\mathbb C}^n \times \mathbb C^* $ so that
$T_{\mathbb C}^n\times \{\text{Id}\}$ is isomorphic to the torus
action on $M$, we get  $\mathbb C^*$-action $\sigma^u$ by
$\{\text{Id}\}\times{\mathbb C^*}$. Hence, we define an equivariant map
$$\pi: M_Q\rightarrow \mathbb C\mathbb P^1$$
satisfying $\pi^{-1}(\infty)=i(M)$.
One can check that
$\mathcal {W}=M_Q\backslash i(M)$ is  a test configuration
for the pair $M$, called a {\it toric degeneration}.

Let  $kP$ be the  polytope  which corresponds to the bundle $-kK_M$.
Let  $B_{k, P}=\mathbb Z^n\cap k\overline P$   be the lattices set  of $kP$. 
Let $d\sigma=\langle \vec{n},x\rangle \, d\sigma_0 $, where $\vec{n}$ is the unit outer  normal vector field, 
and $d\sigma_0$ is the Lebesgue measure on on $\partial P$.  
We need the following lemma.

\begin{lemma}\label{integral-formula}
Let $\phi$ be a continous function on $\overline P$, then
\beq
\sum_{I\in{B_{k,P}}}\phi(I/k)=k^{n}\int_P\phi \,dx +
\frac{k^{n-1}}{2}\int_{\partial P}\phi \,d\sigma +O(k^{n-2}).
\eeq
In particular, 
$$N_k=k^{n}|P|+
\frac{k^{n-1}}{2}|\partial P| +O(k^{n-2}).$$
Note that $\frac{|\partial P|}{|P|}=n$ if $P$ corresponds to $2\pi c_1(M)$.
\end{lemma}

This lemma was proved in \cite{D1} for convex rational PL functions. 
It is easy to see that the formula can be extended to continuous function by approximation arguments.

\begin{proposition} 
Let 
\beqn \label{l-functional} 
\mathcal L(u)=\int_{\partial P}ue^{\theta(x)}  \,d\sigma-\int_P (n+\theta(x))e^{\theta(x)} u\, dx.
\eeqn
Then  for   a toric
degeneration on $M$ induced by a positive rational PL-convex function 
$u$, we have
\beq F_1=\frac{1}{2Vol(P)}\mathcal L(u).
\eeq
\end{proposition}

\begin{proof} 
We consider the space $H^0(\mathcal {W}, \mathfrak {L}^k)$
of holomorphic sections over $\mathcal {W}$. It is well-known that  
$H^0(\mathcal {W}, \mathfrak {L}^k)$ has a basis
$\{S_{I,i}\}$, where $I$ is a lattice in $B_{k,P}$ and 
$0\leq i\leq k(R-u)(I/k)$. By using the exact sequence for large $k$,
$$0\longrightarrow H^0( \mathcal {W}, \mathfrak {L}^k\otimes \pi^*(\vartheta(-1)))
\longrightarrow H^0( \mathcal {W}, \mathfrak {L}^k)
\longrightarrow H^0(M_0, \mathfrak {L}^k) \longrightarrow 0,$$
 $H^0(M_0, \mathfrak L^k)$ has a basis $\{S_{I,k(R-u)(I/k)}|_{M_0}\}_{I\in{B_{k,P}}}$.
According to \cite{ZZ1}, the action $\sigma^X$ induced by $X$ acts on
$S_{I,k(R-u)(\frac{I}{k})}$ with weight $k \theta(I/k)$. 
The action $\sigma^u$ induced by $u$ acts on $S_{I,k(R-u)(\frac{I}{k})}$ with weight
$k(R-u)(I/k)$. Then by Lemma \ref{integral-formula}, it is easy to see that
\beqs
S_1&=& \sum_{I\in{B_{k,P}}}e^{\theta(I/k)} k(R-u)(I/k)\\[7pt]
&=& \int_P e^{\theta(x)}(R-u)\, dx -
\frac{k^{n}}{2}\int_{\partial P}e^{\theta(x)}(R-u)\, d\sigma +O(k^{n-1})
+\frac{1}{2}\left[k^n\int_P e^{\theta(x)}(R-u) \cdot \theta(x) \,dx\right],\\[7pt]
S_2&=&\sum_{I\in{B_{k,P}}}e^{\theta(I/k)} \theta(I/k)\cdot (R-u)(I/k)\\[7pt]
&=&\frac{1}{2}\left[k^n\int_P e^{\theta(x)}(R-u) \cdot \theta(x) \,dx -
\frac{k^{n-1}}{2}\int_{\partial P}e^{\theta(x)}(R-u)\cdot \theta(x) \,d\sigma +O(k^{n-2})\right].
\eeqs
Then 
\beqs
&&-(S_1-S_2)\\
&&=-k^{n+1}\int_P e^{\theta(x)}(R-u)\, dx-\frac{k^n}{2}\left[\int_{\partial P}e^{\theta(x)}(R-u) \,d\sigma-\int_P e^{\langle v,x\rangle}(R-u) \cdot \theta(x) \,dx \right]+O(k^{n-1}).
\eeqs
We have
$$\frac{-(S_1-S_2)}{kN_k}=  F_{0}+ F_{1} k^{-1}+\cdot\cdot\cdot,$$
where
\beqn F_{0}&=&-\frac{1}{|P|}\int_P e^{\theta(x)} (R-u) \,dx,\nonumber\\[4pt]
 F_{1}&=&\frac{1}{2|P|}\left[\int_{\partial P}e^{\theta(x)} (R-u)\,d\sigma-\int_P e^{\theta(x)}(R-u) \cdot \theta(x) \,dx-\frac{|\partial P|}{|P|}\int_P e^{\theta(x)} (R-u)\, dx\right]\nonumber\\[4pt]
\label{F_1}&=&\frac{R}{2|P|}\left[\int_{\partial P}e^{\theta(x)} \,d\sigma
-\int_P  (n+\theta(x))e^{\theta(x)} \,dx\right]
-\frac{1}{2|P|}\left[\int_{\partial P}e^{\theta(x)} u\,d\sigma
-\int_P (n+\theta(x)) e^{\theta(x)} u\,dx\right].
\eeqn
Integrating by parts, we have
\beqs
\int_P e^{\theta(x)}\, dx
&=&\frac{1}{n}\int_{\partial P} e^{\theta(x)}\,d\sigma
-\frac{1}{n}\int_P\theta(x) e^{\theta(x)}\,dx.
\eeqs
Therefore, the coefficient of $R$ in (\ref{F_1}) vanishes, and we have
\beq
 F_{1}=\frac{1}{2|P|}\left[\int_{\partial P}e^{\theta(x)} u\, d\sigma-\int_P (n+\theta(x))e^{\theta(x)} u\, dx\right].
\eeq
\end{proof}

\begin{remark}
As can be seen in the above lemma, the weights of the action
depend on $R$ and $F_0$ also depends on the integer $R$. But $F_1$ is independent of $R$. In particular, $F_0$ is different to $F_1$.
\end{remark}

Since $X$ is the soliton vector field, $\mathcal L(u)=0$ for any linear function $u$.
This implies that $\mathcal L(u)$ is invariant when adding $u$ by a linear function. We call a convex function is normalized at $0\in P$ if $\inf_{P} u=u(0)$.
Let $\mathcal C_\infty$ be the set of smooth convex functions on $\overline P$ and
$\tilde{\mathcal C}_\infty$ be the set of smooth convex functions normalized at $0\in P$.
It is clear that the PL functions can be approximated uniformly by functions in $\mathcal C_\infty$.

\begin{lemma} \label{stable}
There exists a $\lambda>0$ such that
\beq\label{cond1}
\mathcal L(u)\geq \lambda\int_{\partial P}ue^{\theta(x)}\, d\sigma,\ \ u\in \tilde{\mathcal C}_\infty.
\eeq
\end{lemma}

\begin{proof}
We  note that  $\mathcal L(u)$  can be rewritten as
\beq\label{inq2}
\mathcal L(u)=\int_P [(\sum x_iu_i-u)+u ]e^{\theta(x)}\, dx\geq \int_P u e^{\theta(x)}\, dx.
\eeq
By the contradiction,  we suppose that (\ref{cond1}) is not true.  Then
there is a sequence of functions $\{u_{k}\}$ in
$\tilde{\mathcal C}_\infty$ such that
\beq\label{nor}
\int_{\partial P} u_{k} e^{\theta(x)}\,d\sigma=1
\eeq
and
\beq\label{conver}
\mathcal L(u_k)\longrightarrow 0, \text{as} \  \ k\longrightarrow\infty.
\eeq
By (\ref{nor}), there exists a subsequence (still denoted by $\{u_k\}$) of $\{u_k\}$, which converges locally uniformly to a convex function $u_\infty\ge 0$ on $P$.
By (\ref{inq2}) and (\ref{conver}), we have
$$\int_P u_k e^{\theta(x)}\,dx\leq\mathcal L(u_k)\longrightarrow 0.$$
It follows
$$\int_{P} u_\infty e^{\theta(x)}\, dx=0.$$
Hence, we obtain $u_\infty\equiv 0$ in $P$.
On the other hand,
\beqs
\mathcal L(u_k)&=&\int_{\partial P} u^ke^{\theta(x)}d\sigma
- \int_{P} (n + \sum {x_{i}\theta_{i}})u_k e^{\theta(x)}\,dx \\[5pt]
&\longrightarrow& 1-\int_{P} (n + \sum {x_{i}\theta_{i}})u_\infty e^{\theta(x)}\,dx=1>0.
\eeqs
This contradicts with (\ref{conver}). The lemma is proved.
\end{proof}

By Lemma  \ref {stable}, we immediately get

\begin{theorem}
Any toric Fano manifold is modified $K$-stable for toric degenerations.
\end{theorem}

\vskip 20pt

\section{Modified $K$-energy on a toric Fano manifold}

\vskip 10pt

Let  $K_X$  be a one parameter compact subgroup generated by  the image part $\text{Im}(X)$ and denote by  $\mathcal H_X(\omega_g)$ a set of $K_X$-invariant K\"ahler potentials.
In the  study  of K\"ahler-Ricci solitons,  the modified Mabuchi's $K$-energy   $\mu_{\omega_{g}}(\phi)$  plays an important role \cite{TZ2, CTZ}, where
\beqs
\mu_{\omega_{g}}(\phi)
= - \frac{1}{V}\int_{0}^{1}\int_{M}
\dot\phi_{t}[S(\phi_{t}) - n -
tr_{\omega_{\phi_{t}}}(\nabla_{\omega_{\phi_{t}}}X) - X(h_{\omega_{\phi_{t}}} -
\theta_{X}(\omega_{\phi_{t}}))]e^{\theta_{X}(\omega_{\phi_{t}})}\frac{\omega_{\phi_{t}}^n}{n!}\wedge
dt.
\eeqs
Here $\phi \in \mathcal H_X(\omega_{g})$ and $g$ is chosen to be  $K_X$-invariant.
Recall  two Aubin typed functionals  introduced in \cite{Z1},
\beqs
I_{\omega_{g}}(\phi)&=&\frac
{1}{V}\int_{M}\phi(e^{\theta_{X}(\omega_g)}\frac{\omega_g^n}{n!} -
e^{\theta_{X}(\omega_\phi)} \frac{\omega_{\phi}^n}{n!}),\\
J_{\omega_{g}}(\phi) &=& \frac{1}{V}\int_{0}^{1}\int_{M}
\dot\phi_{s}(e^{\theta_{X}(\omega_g)}\frac{\omega_g^n}{n!}  -e^{\theta_{X}(\omega_{\phi_{s}})}\frac{\omega_{\phi_{s}}^n}{n!})\wedge ds,
\eeqs
where $\omega_\phi=\omega_g+\sqrt{-1}\partial\bar\partial \phi$ and
$\phi_s$ is a path in $\mathcal H_X(\omega_g)$. It is known that
\beq\label{ij}
c_1I_{\omega_g}(\phi)\geq J_{\omega_g}(\phi)\geq c_2 I_{\omega_g}(\phi),
\ \forall \ \phi\in\mathcal H_X(\omega_g),
\eeq
for two positive constants $c_1$, $c_2$.
Then $\mu_{\omega_{g}}(\cdot)$ can also be written as
\beqn\label{kenergy}
\mu_{\omega_{g}}(\phi)  &=& - \frac{1}{V}\int_{M}
\log\left(\frac{e^{\theta_{X}(\omega_\phi)}\omega_{\phi}^{n}}{e^{\theta_{X}(\omega_g)}\omega_{g}^{n}}\right)
e^{\theta_{X}(\omega_\phi)}\frac{\omega_{\phi}^{n}}{n!} - (I_{\omega_{g}}(\phi) -
J_{\omega_{g}}(\phi)) \nonumber\\
&& + \frac{1}{V}\int_{M}(h_{g} -
\theta_{X}(\omega_g))(e^{\theta_{X}(\omega_g)}\frac{\omega_g^n}{n!}  -
e^{\theta_{X}(\omega_\phi)}\frac{\omega_{\phi}^{n}}{n!}).
\eeqn

\begin {definition} \label{proper-definition}
Let $(M,g)$ be a  Fano manifold $M$.
Let  $G$ be a reductive  subgroup of  automorphisms group ${\rm Aut}(M)$ which contains $K_X$.
We call  $\mu_{\omega_g}(\phi)$ proper  modulo $G$
if there is a continuous function $p(t)$ in $\mathbb R$ with the property
$$ \lim_{t\to+\infty} p(t)=+\infty,$$
such that
\beq\label{properdef}
\mu_{\omega_g}(\phi)\ge \inf_{\sigma\in G} p(I_{\omega_g}(\phi_{\sigma})),
\eeq
where $\phi_{\sigma}$ is defined by
$$\omega_g+\sqrt{-1}\partial\bar{\partial}\phi_{\sigma}
=\sigma^*(\omega_g+\sqrt{-1}\partial\bar{\partial}\phi).$$
\end{definition}

The  properness of   $\mu_{\omega_g}(\phi)$ is a sufficient  condition for the existence of K\"ahler-Ricci solitons due to the following lemma.

\begin{lemma}\label{sufficient-condition}
 Suppose that $\mu_{\omega_g}(\phi)$  is proper  modulo   a reductive  subgroup $G$ of  automorphisms group ${\rm Aut}(M)$ which contains $K_X$. Then
$M$ admits a  K\"ahler-Ricci soliton.
\end{lemma}

Lemma \ref{sufficient-condition}  was   proved by   using  K\"ahler-Ricci flow as in  \cite{TZ3, Z2, BB} and can be   also  proved by uisng the continuity  method as in \cite {CTZ,T2}.

Let $P_{g}=\triangle_g+X(\cdot)$ be a linear elliptic operator defined on
the space $$\mathcal N_X=\{u\in C^\infty(M)| ~{\rm  Im}(X(u))=0\},$$
associated to $\omega_g$ and a holomorphic on $M$.
$P_{g}$ is a self-adjoint elliptic operator on $\mathcal N_X$ with respect to the inner product,
$$(\phi,\psi)=\int_{M}\phi\psi e^{\theta_X(\omega_g)}\frac{\omega_g^n}{n!} .$$
The following lemma  shows  that the properness given in Definition \ref{proper-definition} coincides with one  as  defined  in \cite{T1, CTZ}, when $g$ is a K\"ahler-Ricci soliton.

\begin{lemma}
Suppose that $M$ admits a  K\"ahler-Ricci soliton  $g_{KS}$.
 Then the  modified $K$-energy $\mu_{\omega_{KS}}(\phi)$
is proper with respect to $X$ modulo ${\rm Aut}^0(M)$
iff there is a continuous function $\bar p(t)$ in $\mathbb R$ with
the property
\beqs
\lim_{t\to+\infty} \bar p(t)=+\infty
\eeqs
such that
\beq\label{properdef1}
\mu_{\omega_{KS}}(\phi) \geq \bar p(I_{\omega_{KS}}(\phi)), \ \ \forall  \ \phi\in\Lambda_1^{\bot}(M,g_{KS}),
\eeq
 where  $\Lambda_1(M,g_{KS})$ denotes the first non-zero eigenfunctions space
for the operator $P_{g_{KS}}$ associated to the metric $g_{KS}$, i.e.,
$\Lambda_1(M,g_{KS})={\rm Ker} (P_{g_{KS}}+I)$.
\end{lemma}

\begin{proof}
First we prove the necessary part of the lemma.
We  choose the K\"ahler-Ricci soliton metric $g_{KS}$ as an initial metric.
Then we induce a functional on
$\text{Aut}^0(M)$ for any $\phi\in\mathcal H_X(\omega_{KS})$ by
$$\Phi(\sigma)=I_{\sigma^*\omega_\phi}(-\phi_{\sigma})-J_{\sigma^*\omega_\phi}(-\phi_{\sigma}),$$
 where $\phi_{\sigma}$ is an induced K\"ahler potential defined by
$$\omega_{KS}+\sqrt{-1}\partial\bar\partial\phi_\sigma.$$
$I_{\sigma^*\omega_\phi}(\psi)$ and $J_{\sigma^*\omega_\phi}(\psi)$
are functionals $I(\psi)$ and $J(\psi)$ respectively while the initial metric $\omega_{KS}$ is replaced by $\sigma^*\omega_\phi$. According to \cite{TZ1},
one can show that there exists a $\tau\in\text{Aut}^0(M)$ such that
$$\Phi(\tau)=\inf_{\sigma\in\text{Aut}^0(M)}\Phi(\sigma)$$
and consequently $\phi_{\tau}\in\Lambda_1^{\bot}(M,g_{KS})$.
In fact,  from the proof of uniqueness of K\"ahler-Ricci solitons in \cite{TZ1}
it can be  proved that
$\phi\in\Lambda_1^{\bot}(M,g_{KS})$ iff
$$I_{\omega_\phi}(-\phi)-J_{\omega_\phi}(-\phi)
=\inf_{\sigma\in\text{Aut}^0(M)}\Phi(\sigma).$$
Thus by the assumption (\ref{properdef}) and relation (\ref{ij}),  for any
$\phi\in\Lambda_1^{\bot}(M,g_{KS})$, we have
\beqs
\mu_{\omega_{KS}}(\phi)
&\geq& \inf_{\sigma\in\text{Aut}^0(M)}p(I_{\omega_{KS}}(\phi_{\sigma}))\\[4pt]
&=&\inf_{\sigma\in\text{Aut}^0(M)}p(I_{\sigma^*\omega_\phi}(-\phi_\sigma))\\[4pt]
&\geq& \inf_{\sigma\in\text{Aut}^0(M)} \tilde p(\Phi(\sigma))\notag\\[4pt]
&=&\tilde p(I_{\omega_\phi}(-\phi)-J_{\omega_\phi}(-\phi))\\
&\geq& \tilde p(I_{\omega_\phi}(-\phi))\\
&=& \tilde p(I_{\omega_{KS}}(\phi)),
\eeqs
where $\tilde p(t)$ is another continuous function in $\mathbb R$
which satisfying (\ref{cond1}).

Next we prove the sufficient part.  We note that  the modified $K$-energy is invariant under ${\rm Aut}^0(M)$ \cite{TZ2}.
Then by the discussion at last paragraph, for any $\phi\in\mathcal H_X(\omega_{KS})$, we can choose a $\tau\in\text{Aut}^0(M)$ such that
$\phi_\tau\in\Lambda_1^{\bot}(M,g_{KS})$ and
$$\mu(\phi)=\mu(\phi_\tau).$$
Thus by (\ref{properdef1}), we get
$$\mu(\phi)\geq \bar p(I(\phi_\tau))\geq\inf_{\sigma\in\text{Aut}^0(M)}
\bar p(I(\phi_\sigma)).$$
\end{proof}

The converse of Lemma \ref{sufficient-condition} was conjectured by
Tian in sense of (\ref{properdef1}) in  the  case of K\"ahler-Einstein metrics \cite{T1}
and was proved by him under  the assumption that there  is no any holomorphic vector field on $M$.
Thus one may believe  that the converse of Lemma \ref{sufficient-condition}
is also true as a generalization of  Tian's conjecture for the case of K\"ahler-Ricci solitons \cite{CTZ}.
In  this section  we give a positive answer in  the case of  toric Fano manifolds. Namely, we shall  prove

\begin{theorem}\label{WZZ-2}
On a toric Fano manifold,  the  modified $K$-energy is proper
for  toric invariant K\"ahler potentials modula toric action.
\end{theorem}

Theorem \ref{WZZ-2} has been proved under  the assumption that the Futaki invariant vanishes \cite{ZZ1}.  In the following, we alway assume that $M$ is  a toric Fano manifold.

\subsection{The reduction of modified $K$-energy}

Denote
$\mathcal H_{G_0}(\omega_g)\subset \mathcal H_{X}(\omega_g) $ to  be the set of $G_0$-invariant K\"ahler potentials.  Then $\mathcal H_{G_0}(\omega_g)$ is equal to the set
$$\{\phi\in C^{\infty}({\mathbb R^n})| ~|\phi|<\infty~{\rm and}~\varphi_0+\phi~{ \rm is~  strictly ~convex}\}.$$
 By using the Legendre transformation $\xi=(D\varphi_0)^{-1}(x)$, one sees
that the function (Legendre  dual function) defined by
$$u_0(x)=\langle \xi,D\varphi_0(\xi) \rangle -\varphi_0(\xi)=\langle \xi(x),
x\rangle-\varphi_0(\xi(x)),\ \forall\ x\in P$$
is strictly convex. Set the space of symplectic functions by
$$\mathcal C =\{u=u_0+f\ | \  u\
\text{is a  strictly convex function in }\ P,\ f\in C^{\infty}(\overline{P})\}.$$
It was shown in  \cite{Ab1} that  there is a bijection between $\mathcal C$ and $\mathcal H_{G_0}(\omega_g)$.

\begin{proposition}
Let  $\phi\in \mathcal H_{G_0}(\omega_g)$ and  $u$
be the Legendre  dual function of $\varphi_0+\phi$.
Then the modified $K$-energy is given by
\beqn\label{reduced-energy}\mu_{\omega_g}(\omega_\phi)=\frac{(2\pi)^{n}}{V} \mathcal F(u)+C,
\eeqn
where
\beq\label{reduced}
\mathcal F(u) =
- \int_{P} \log\det(u_{ij})e^{\theta(x)}\,dx + \mathcal L(u),
\eeq
and $C$ is a constant.
\end{proposition}

\begin{proof}
By (\ref{kenergy}), a direct computation shows
\beqn\label{redu}\nonumber
&&\mu_{\omega_{g}}(\phi)\\
&&= \frac{1}{V}\int_{M}
\log\left(\frac{e^{\theta_{X}(\phi)}\omega_{\phi}^{n}}{e^{\theta_{X}(\omega_g)}\omega_{g}^{n}}\right)
e^{\theta_{X}(\omega_\phi)}\frac{\omega_{\phi}^{n}}{n!}-\left[\frac{1}{V}\int_{0}^{1}\int_{M} \dot\phi_{t}
e^{\theta_{X}(\omega_{\phi_{t}})}\frac{\omega_{\phi_{t}}^{n}}{n!}\wedge dt -
\frac{1}{V}\int_{M}\phi e^{\theta_{X}(\omega_\phi)}\frac{\omega_{\phi}^{n}}{n!}\right]\nonumber\\
&&  \ \ \ -  \frac{1}{V}\int_{M}(h_{g} -
\theta_{X}(g))e^{\theta_{X}(\omega_\phi)}\frac{\omega_{\phi}^{n}}{n!} +
\frac{1}{V}\int_{M}(h_{g} -
\theta_{X}(\omega_g))e^{\theta_{X}(\omega_g)}\frac{\omega_{g}^{n}}{n!}\nonumber\\
&&= \frac{1}{V}\int_{M}
\log\left(\frac{\omega_{\phi}^{n}}{\omega_{g}^{n}} e^{\phi - h_{g}}\right)
e^{\theta_{X}(\omega_\phi)}\frac{\omega_{\phi}^{n}}{n!}- \frac{1}{V}\int_{0}^{1}\int_{M} \dot\phi_{t}
e^{\theta_{X}(\omega_{\phi_{t}})}\frac{\omega_{\phi_{t}}^{n}}{n!}\wedge dt+
\frac{1}{V}\int_{M}\theta_{X}(\omega_\phi) e^{\theta_{X}(\omega_\phi)}\frac{\omega_{\phi}^{n}}{n!}\nonumber\\
&& \ \ \ +\frac{1}{V}\int_{M}(h_{g} -\theta_{X}(\omega_g))e^{\theta_{X}(\omega_g)}\frac{\omega_{g}^{n}}{n!}\nonumber\\
&&= \frac{1}{V}\int_{M}
\log\left(\frac{\omega_{\phi}^{n}}{\omega_{g}^{n}} e^{\phi - h_{g}}\right)
e^{\theta_{X}(\omega_\phi)}\frac{\omega_{\phi}^{n}}{n!}- \frac{1}{V}\int_{0}^{1}\int_{M} \dot\phi_{t}
e^{\theta_{X}(\omega_{\phi_{t}})}\frac{\omega_{\phi_{t}}^{n}}{n!}\wedge dt +{\rm const.}
\eeqn
On the other hand,
$$h_{g} = - \varphi_{0} - \log \det (\varphi_{0ij})+C.$$
 Then
$$\frac{\omega_{\phi}^{n}}{\omega_{g}^{n}} e^{\phi - h_{g}}
= C\det (\varphi_{ij})e^{\varphi}.$$
It follows
\beqn \label{determinent}
&&\int_{M}\log\left(\frac{\omega_{\phi}^{n}}{\omega_{g}^{n}} e^{\phi - h_{g}}\right)
e^{\theta_{X}(\omega_\phi)}\frac{\omega_{\phi}^{n}}{n!}\nonumber\\
&&=(2\pi)^n \left[\int_{\mathbb R^n}
\log\det(\varphi_{ij}) e^{X(\varphi)}\det(\varphi_{ij})\,d\xi+\int_{\mathbb R^n}  \varphi e^{X(\varphi)}\det(\varphi_{ij})\,d\xi \right].
\eeqn
By using the relations
\beqs
\varphi=\sum_{i=1}^n {x_{i}u_{i}} - u,\
\det(\varphi_{ij})\,d\xi= dx,\ \
\dot\phi_t= -\dot u_t,
\eeqs
where  $\phi_t$ is a path in $\mathcal H_{G_0}(\omega_g)$ and $u_t$ is the symplectic potential
of $\varphi_t=\varphi_0+\phi_t$, we also get
 \beqn\label{path-integral}
\int_{0}^{1}\int_{M} \dot\phi_{t}
e^{\theta_{X}(\omega_{\phi_{t}})}\frac{\omega_{\phi_{t}}^{n}}{n!}\wedge dt
&&=(2\pi)^n
\int_{0}^{1}\int_{\mathbb R^n} \dot\phi_{t}
e^{X(\varphi_{t})}\det({\varphi_{t}}_{ij})\,d\xi \,\wedge dt \nonumber\\
&&=-(2\pi)^n\int_{P} ue^{\theta(x)}dx + {\rm const.},
\eeqn
\beqn\label{determinent-3} &&\int_{\mathbb R^n}
\log\det(\varphi_{ij}) e^{X(\varphi)}\det(\varphi_{ij})\,d\xi+\int_{\mathbb R^n}  \varphi e^{X(\varphi)}\det(\varphi_{ij})\,d\xi \nonumber \\
&&=-\int_{P}\log\det(u_{ij}) e^{\theta(x)}\,dx+
\int_{P} (\sum_{i=1}^n {x_{i}u_{i}} - u)e^{\theta(x)}\,dx.
\eeqn
Hence inserting (\ref{determinent})-(\ref {determinent-3}) into (\ref{redu}), we obatin
\beqs
&&\mu_{\omega_{g}}(\phi)
= \frac{(2\pi)^{n}}{V}\left[-\int_{P}\log\det(u_{ij}) e^{\theta(x)}\,dx
+\int_{P} \sum_{i=1}^n {x_{i}u_{i}} e^{\theta(x)}\,dx\right] + C.
\eeqs
Integrating by parts, we  deduce (\ref{reduced-energy}) immediately.
\end{proof}

\subsection{Properness of $\mathcal F(u)$}

In this subsection, we show the properness of $\mathcal F(u)$ by similar arguments as in \cite{D1, ZZ1}.
First, we have

\begin{lemma}\label{determinent-estimate}
There exists a constant $C>0$ such that for any $u\in\mathcal C_\infty$, it holds
\beq\label{inq1}
\int_{P} \log\det(u_{ij})e^{\theta(x)}\,dx\leq \mathcal L_B(u)+C
\eeq
where $B=(u_0)_{ij}^{ij}+2(u_0)_i^{ij}\theta_j+(u_0)^{ij}\theta_i\theta_j$ is a bouned function,
and
\beq\label{bfunctional}
\mathcal L_B(u)=\int_{\partial P}ue^{\theta(x)}\,d\sigma+\int_P Bue^{\theta(x)}\,dx.
\eeq
\end{lemma}

\begin{proof}
Let $f=u-u_0$. By the convexity of $-\log\det$,
we have
$$\log\det (u_{ij})\leq \log\det ((u_0)_{ij})+(u_0)^{ij}f_{ij}.$$
For any $\delta > 0$, let $P_{\delta}$
be the interior polygon with faces parallel to those of $P$
separated by distance $\delta$, then $f$ is smooth over the
closure of $P_{\delta}$.

Integrating by parts,
$$\int_{P_{\delta}}(u_0)^{ij}f_{ij} e^{\theta(x)}\,dx = \int_{\partial P_{\delta}}
(u_0)^{ij}f_{i}n_{j} e^{\theta(x)} \,d\sigma_0 -
\int_{P_{\delta}}(u_0)^{ij}_{j}f_{i}e^{\theta(x)}\,dx -
\int_{P_{\delta}} (u_0)^{ij}f_{i}\theta_{j}e^{\theta(x)}\,dx.$$
Integrating by parts for the last two terms again,
we have
\beqs
\int_{P_{\delta}}(u_0)^{ij}f_{ij} e^{\theta(x)}\,dx
&=& \int_{\partial P_{\delta}} (u_0)^{ij}f_{i}n_{j} e^{\theta(x)}\,
d\sigma_0 - \int_{\partial P_{\delta}} (u_0)^{ij}_{j}n_{i}f e^{\theta(x)} d\sigma_0 - \int_{\partial P_{\delta}} (u_0)^{ij}n_{i}\theta_{j}f e^{\theta(x)} \,d\sigma_0 \\[3pt]
& &+ \int_{P_{\delta}} ((u_0)^{ij}_{ij}+(u_0)^{ij}_{j}\theta_{i}+
(u_0)^{ij}_{i}\theta_{j}+(u_0)^{ij}\theta_{i}\theta_{j})fe^{\theta(x)}\,dx.
\eeqs
Note that
$$ (u_0)^{ij}n_{j}\, d\sigma_0\to 0,\  \ -(u_0)^{ij}_{j}n_{i}d\sigma_0\to \, d\sigma$$
as $\delta \rightarrow 0$ \cite{D1, D2}.
Then
\beqs
\int_{\partial P_{\delta}} (u_0)^{ij}f_{i}n_{j} e^{\theta(x)}\,d\sigma_0,\
 \int_{\partial P_{\delta}} (u_0)^{ij}n_{i}\theta_{j}f
e^{\theta(x)} \,d\sigma_0 \longrightarrow 0,
\eeqs
and
\beqs
\int_{\partial P_{\delta}} (u_0)^{ij}_{j}n_{i}f e^{\theta(x)} \,d\sigma_0 \longrightarrow \int_{\partial P} fe^{\theta(x)} \,d\sigma
\eeqs
as $\delta \rightarrow 0$. In conclusion,
\beqs
\int_{P}(u_0)^{ij}f_{ij} e^{\theta(x)}\,dx
&=& \int_{\partial P} fe^{\theta(x)} \,d\sigma+ \int_{P} Bf e^{\theta(x)}\,dx.
\eeqs
Hence,
\beqs
&&\int_{P} \log\det(u_{ij})e^{\theta(x)}\,dx\\
&\leq& \int_{\partial P}ue^{\theta(x)}\,d\sigma+\int_P Bue^{\theta(x)}\,dx
+\int_{\partial P}u_0e^{\theta(x)}\,d\sigma-
\int_P Bu_0e^{\theta(x)}\,dx+\int_P\log\det ((u_0)_{ij})e^{\theta(x)}\,dx\\[3pt]
&=&\int_{\partial P}ue^{\theta(x)}\,d\sigma+\int_P Bue^{\theta(x)}\,dx+const.
\eeqs
\end{proof}

\begin{remark} 
As in \cite{ZZ2}, Lemma \ref{determinent-estimate} can be extended for any $u\in \mathcal C_\star$,
where
$$\mathcal C_\star=\{u |~ \text{$u$ is convex and satisfies $\int_{\partial P} u\, d\sigma<\infty$}\}.$$  
\end{remark}

Denote
\beq\label{hfunctioonal}
H(u)=\int_P ue^{\theta(x)}\, dx.
\eeq

\begin{proposition}\label{f-proper}
For any $0<\delta<1$, there exsits $C_\delta>0$ such that
\beq
\mathcal F(u)\geq \delta H(u)-C_\delta, \ \ \forall u\in \tilde{\mathcal C}_\infty.
\eeq
\end{proposition}

\begin{proof}
First, we compute the difference of $\mathcal L$(u) and $\mathcal L_B(u)$
\beqs
|\mathcal L(u)-\mathcal L_B(u)|
&=& \left|\int_{P}(n+\sum\theta_ix_i+B)u e^{\theta(x)}\,dx\right| \\
&\leq & C'\int_{P} u e^{\theta(x)}\,dx\\
&\leq & (1+\delta)C_0C'\int_{\partial P} u e^{\theta(x)}\,d\sigma-\delta C'\int_{P} u e^{\theta(x)}\,dx
\eeqs
where $C'=\|n+\sum\theta_ix_i+B\|_{L^\infty}$. Note
\beqs
\int_{P}u e^{\theta(x)}\,dx \leq C_0\int_{\partial P}u e^{\theta(x)}\,d\sigma, ~\forall~ u\in \tilde{\mathcal C}_{\infty}.
\eeqs
Then by (\ref{cond1}), it follows
\beqs
|\mathcal L(u)-\mathcal L_B(u)|\leq & \frac{(1+\delta)C_0C'}{\lambda}\mathcal L(u)-\delta C'\int_{P} u e^{\theta(x)} dx.
\eeqs
Thus
$$\left(1+\frac{(1+\delta)C_0C'}{\lambda}\right)\mathcal L(u)\geq \mathcal L_B(u)+\delta C'\int_{P} u e^{\theta(x)}\,dx.$$
Now let $r=\left(1+\frac{(1+\delta)C_0C'}{\lambda}\right)^{-1}$, we get
$$\mathcal L(u)\geq \mathcal L_B(ru)+r\delta C'\int_{P} u e^{\theta(x)}\,dx.$$
Applying the inequality (\ref{inq1}) to $ru$, we obtain
$$-\int_{P} \log\det(u_{ij})e^{\theta(x)}\,dx\geq -\mathcal L_B(ru)-C+n\log r.$$
Hence,
\beqs
\mathcal F(u)
\geq r\delta C'\int_{P} u e^{\theta(x)}\,dx-C+n\log r.
\eeqs
\end{proof}

\vskip 15pt

\subsection{Properness of $\mu_{\omega_g}(\cdot)$}

In this subsection, we show that the properness of $\mathcal F(u)$ in the above subsection is equivalent to
the properness of $\mu_{\omega_g}(\phi)$.  We need a lemma as follows.

\begin {lemma}\label{h-functional}
There exists $C>0$ such that
$$|J_{\omega_g}(\tilde{\phi})- H(u_{\tilde\phi})|\leq C, ~\forall~ \phi\in \mathcal H_{G_0}(\omega_g),$$
where $\tilde\phi=\phi_\sigma$ is a normalization of $\phi$ after a transformation $\sigma\in T$
so that
$$(\psi_0+\tilde \phi)(0)=0,~ D(\psi_0+\tilde \phi)(0)=0.$$
\end {lemma}

\begin {proof}
By the relation $\dot{\phi_t} = - \dot{u_{t}}$,  it is easy to see
\beqs
J_{\omega_g}(\phi)=\frac{1}{V}\int_{M} \phi e^{\theta_X(g)}\frac{\omega^{n}_g}{n!} + H(u_{\phi})-H(u_{0}),~\forall ~\phi\in \mathcal H^{G_0}_{X}(\omega_g).
\eeqs
In particular,
$$J_{\omega_g}(\tilde{\phi})- H(u_{\tilde{\phi}}) =
\frac{1}{V}\int_{M}\tilde{\phi} e^{\theta_X(g)}\omega^{n}_g-H(u_{0}).$$
We claim that
$$\left|\frac{1}{V}\int_{M}\tilde{\phi} e^{\theta_X(g)}\frac{\omega^{n}_g}{n!}\right| \le C$$
for some uniform constant $C$.

Let $G(p, p')$ be the Green function of $P_{\omega_g}$ so that
$$\int_M G(p,\cdot)e^{\theta_X(g)}\frac{\omega_g^n}{n!} =0.$$
It is  proved  in \cite{CTZ} that  there  exists a  $C>0$ depending only on $g$  such that
$$G(p, p')\geq -C.$$
 Then applying the Green's formula to potential $\tilde\phi$,  we have
\beqn\label{inq3}
\tilde{\phi}(x)&=&\frac{1}{V}\int_{M}\tilde{\phi}e^{\theta_X(g)}\frac{\omega^{n}_g}{n!}-
\int_M G(x,\cdot)(\triangle\tilde\phi(\cdot)+X(\tilde\phi))e^{\theta_X(g)}\frac{\omega^{n}_g}{n!} \nonumber\\
&\leq&\frac{1}{V}\int_{M}\tilde{\phi}e^{\theta_X(g)}\frac{\omega^{n}_g}{n!}
+C_0,
\eeqn
where $C_0$ is a uniform constant.
The second inequality follows from $\triangle \tilde\phi\geq -n$ and that $X(\tilde\phi)$ is
uniformly bounded \cite{Z1}.
Thus
\beq\label{inq4}
\frac{1}{V}\int_{M}\tilde{\phi}e^{\theta_X(g)}\frac{\omega^{n}_g}{n!} \geq
\sup_M{\{\tilde{\phi}\}} - C_{0}= \sup_{\Bbb R^n}{\{\tilde{\phi}\}}
- C_{0}.
\eeq
Set
$$\Omega_N =\{  \xi\in M|~\tilde{\phi}(\xi)\leq \sup_M {\{\tilde{\phi}\}} - N\}.$$
Note that
\beqs
\frac{1}{V}\int_{M}\tilde{\phi}e^{\theta_X(g)}\frac{\omega^{n}_g}{n!}
&=& \frac{1}{V} \int_{M \cap \Omega_N}\tilde{\phi}e^{\theta_X(g)}\frac{\omega^{n}_g}{n!}+
\frac{1}{V}\int_{M \setminus \Omega_N}\tilde{\phi}e^{\theta_X(g)}\frac{\omega^{n}_g}{n!}\\
&\leq&  \frac{1}{V}[(\sup_M{\{\tilde{\phi}\}} - N) \widetilde{\rm Vol}(M\cap \Omega_N) + \sup_M {\{\tilde{\phi}\}}\widetilde{\rm Vol}(M\setminus \Omega_N)]\\
&=& \sup_M{\{\tilde{\phi}\}} - \frac{N\cdot \widetilde{\rm Vol}(M \cap \Omega_N)}{\rm Vol(M)}.
\eeqs
Here $\widetilde{\rm Vol}(\cdot)$ represents the volume with form $e^{\theta_X(g)}\frac{\omega^{n}_g}{n!}$.
Hence by (\ref{inq4}),  we derive
\beq\label{inq5}
\widetilde{\rm Vol}(M \cap \Omega_N)
\leq \frac{C_{0}{\rm Vol}(M)}{N}=\frac{C_0V}{N}\to 0,
\eeq
 as $N\to \infty$.

On the other hand, by the normalization, we have
$\tilde\phi(0)=-\psi_0(0)$. Note $D(\tilde \phi+\psi_0)\in P$. Then we have
$$|D\tilde \phi|\le 2\sup\{|p|: ~p\in P\}$$
and
$$\tilde \phi(x)\le \tilde \phi(0)+2r\sup\{|p|: ~p\in P\}\le C(r),~\forall~x\in B_r(0),$$
where $C(r)$ depends only on the radius $r$ of ball $B_r(0)$ centered at the original.
Since  the volume of domain $B_1(0)\times (0, 2\pi)^n\subset M$  associated the metric
$\omega_g=\sqrt{-1}\partial\overline\partial \psi_0$
is bigger than some number $\epsilon>0$, by (\ref{inq5}),
it is easy to see that there is at least a point $x_0 \in B_1(0)$ such that
$$\tilde\phi(x_0)\ge \sup_{M}\tilde\phi-N$$
if $N$ is sufficiently large. Hence
$$ \sup_{M}\tilde\phi\leq N+C,$$
and consequently
$$\frac{1}{V}\int_{M}\tilde{\phi}e^{\theta_X(g)}\frac{\omega^{n}_g}{n!}\le N+C.$$
By (\ref{inq4}), we also get
$$\frac{1}{V}\int_{M}\tilde{\phi}e^{\theta_X(g)}\frac{\omega^{n}_g}{n!}\geq \tilde\phi(0)-C_0=-\psi(0)-C_0.$$
Therefore the claim  is true and the lemma is proved.
 \end{proof}

\begin{theorem}\label{properness-energy}
There exist numbers $\delta>0$ and $C$ such that
 \beqn
\mu_{\omega_g}(\phi) \geq
\delta \inf_{\tau\in {T}}{I_{\omega_g}(\phi_\tau)} - C,  ~\forall~\phi\in \mathcal H_{G_0}(\omega_g).
\eeqn
 In particular, $\mu(\phi)$ is proper  for any  $\phi\in \mathcal H_{G_0}(\omega_g)$ modulo $G_0$.
\end{theorem}

\begin{proof}
Let  $\phi\in \mathcal H_{G_0}(\omega_g)$.
Then there exists a $\sigma \in T$ such that the Legendre function $u_{\phi_{\sigma}}$
associated to $\phi_{\sigma}$ is belonged to $\tilde{\mathcal {C}}_\infty$.
By Proposition \ref{f-proper}, we see that
$$\mu_{\omega_g}(\phi_{\sigma}) \geq \delta H(\phi_\sigma) -
C_{\delta}.$$
Note that $\mu_{\omega_g}(\phi)=\mu_{\omega_g}(\phi_{\sigma})$.
Thus by Lemma \ref{h-functional}, we get
\beqs
\mu_{\omega_g}(\phi)=\mu_{\omega_g}(\phi_{\sigma}) &\geq& \delta J_{\omega_g}(\phi_\sigma)
- C_{\delta}'\\[4pt]
&\geq& \delta \inf_{\tau\in {T}}J_{\omega_g}(\phi_\tau) -
C_{\delta}'\\
&\geq&\frac{\delta}{n+1} \inf_{\tau\in {T}}I_{\omega_g}(\phi_\tau) -
C_{\delta}'.
\eeqs
 Here at the last inequality we used the relation (\ref{ij}).
\end{proof}

\begin{remark}Theorem \ref{properness-energy} seems to overlap Theorem 4.5 in \cite{BB} where  the   Aubin-Ding  typed functional is studied instead of  modified $K$-energy by using the  geodesic theory for K\"ahler potentials.  \footnote{We are indebted to Chi Li  for telling us  Berman-Berndtsson's results in  \cite{BB}.}
\end{remark}

As an application of Theorem \ref{properness-energy}  together with   Lemma \ref{sufficient-condition}, we give a new proof  of  the following  Wang-Zhu Theorem \cite{WZ}.

\begin{theorem}\label{wang-zhu} There exists a K\"ahler-Ricci soliton on any toric Fano manifold.

\end{theorem}

\vskip 20pt

\section{Generalization to conic metrics case}
\vskip 10pt

Singular K\"ahler-Ricci solitons on toric manifolds have been  extensivedly studied by \cite{SZ, Le, BB, DGSW}, etc.
In this section, we generalize the discussion in former sections to give a new approach by showing
the properness of modified Log $K$-energy.

Let $M$ be a toric Fano manifold and $K_M$ is the canonical line bundle. Let $\{D_i\}_{i=1}^d$  be the toric divisors corresponding to the faces of the moment polytope.
Suppose $\beta>0$ and $D=\sum_{i=1}^d (1-\beta_i) D_i\in |-(1-\beta)K_M|$
be an effective $\mathbb R$-divisor with strictly normal crossing support
and $0<\beta_i\leq 1$ for each $i$.
A conical K\"ahler metric $g$ on $M$ with angle $2\pi\beta_i$ along $D_i$ is a closed positive $(1,1)$ current in $2\pi c_1(M)$, which is  a smooth K\"ahler metric  $\omega_D$ in $M\setminus D$ and satisfies:
for any $p\in D$, there is a coordinates neighborhood $U$ with local holomorphic coordinates $(z^1, \cdot\cdot\cdot, z^n)$ of $p$ such that $D\cap U=\{z^i=0, 1\leq i\leq r\}$ and the metric is  asymptotically equivalent along the model conic metric
$$\sqrt{-1}\sum_{j=1}^r|z^j|^{2\beta_j-2}dz^j\wedge d\bar z^j
+\sqrt{-1}\sum_{j=r+1}^{n}dz^j\wedge d\bar z^j.$$
One can check that  the Guillemin metric $\omega_g=\omega_D= \sqrt{-1}\partial\overline\partial \varphi_0$  induced by symplectic potential
$$u_0=\sum_i  \beta_i^{-1}l_i\log l_i$$
  is a conical K\"ahler metric (cf.  \cite{Ab2, Le,  DGSW}), where $l_i=1-\langle \ell_i, x\rangle$. In fact, if we  let a set of  symplectic potentials
 $$\mathcal C_{\beta} =\{u=u_0+f\ | \  u\
\text{is a  strictly convex function in }\ P,\ f\in C^{\infty}(\overline{P})\},$$
then there is a one-to-one  correspondence between  $\mathcal C_{\beta} $ and
$\mathcal H_D^{G_0}(\omega_{g}),$  where  $\mathcal H^{G_0}_D(\omega_{g})$ consists all $G_0$-invariant K\"ahler potentials which are  asymptotically equilvalent to $\omega_D$.

Let $s_{i}$ be the defining section of $D_i$ and $h_i$  a Hermitian metric on $D_i$.
Denote  $\|s_{i}\|^2$  the norm of $s_{i}$.  Then $h=\otimes_{i=1}^d h_i^{1-\beta_i}$ defines a  Hermitian metric on  $D$ and gives a norm $ \|s_D\| $  for  the defining section  $s=\otimes_{i=1}^d s^{1-\beta_i}$ of $D$.
 By the Poincare-Lelong identity,
$$\sqrt{-1}\partial\bar\partial\log \|s_{i}\|^2=-c_1([D_i], h_i)+\{D_i\},$$
where $\{D_i\}$ denotes the current of integration along $D_i$, we have
\beq\label{section-metric}
\sqrt{-1}\partial\bar\partial \log \|s_D\|^2=-c_1([D], h)+\{D\}.
\eeq
(\ref{section-metric}) implies that there exists $\tau\in\mathbb R^n$ such that
\beqn\label{tau-condition}\log \|s_D\|^2=-(1-\beta)\varphi_0-\beta\langle\tau, \xi\rangle+const.
\eeqn

A conical K\"ahler metric $\omega$ with $2\pi \beta_i$ angle  along each  $D_i$   is called a {\it conical K\"ahler-Ricci soliton} if there  is a holomorphic vector field $X$ on $M$ for some $\beta\in (0,1]$  such that
\beqn\label{conical-kr-soliton}
{\rm Ric}(\omega)-\beta\omega-\{D\}=L_X\omega.
\eeqn
We will investigate   a solution   of   (\ref {conical-kr-soliton})  in  $\mathcal H^{G_0}_D(\omega_{g})$.  Let  $X=\sum \theta_\alpha \zeta_\alpha$, where $\zeta_\alpha$ is a basis of Lie algebra
  $\eta_{T}$ of $T$.  Then a  lemma in \cite{ DGSW} shows that
$\tau=(\tau_1,...,\tau_n)$  in (\ref{tau-condition}) is uniquely determined by relation,
\beq\label{vanishing-condition}
\tau_\alpha=\frac{\int_P x_\alpha e^{\theta(x)}dx}{\int_Pe^{\theta(x)}dx}, \ \ \alpha=1, ..., n.
\eeq
Moreover, $\beta_i=1-\beta l_i(\tau), $ $i=1,...,d$.

Let  $\mathcal H_{X, D}(\omega_{g})$  be a  class of $K_X$-invariant functions  $\phi\in C^{2,\alpha}(M)$  such that  $\omega_g+ \sqrt{-1}\partial\overline\partial \phi$ are conical metrics  with $2\pi \beta_i$ angle  along each  $D_i$.   Following \cite{LS} (also see \cite {Be1, Li}), we consider the following modified  Log  $K$-energy functional
on $\mathcal H_{X, D}(\omega_{g})$,
\beqn\label{logkenergy}
\mu_{\omega_{g}, D}(\phi)  &=&\mu_{\omega_g}(\phi)+(1-\beta)(I_{\omega_g}(\phi)-J_{\omega_g}(\phi))
+\int_M \log\|s_D\|^2e^{\theta_X(\omega_\phi)}\frac{\omega^{n}_{\phi}}{n!}.
\eeqn
It is easy to check that a soultion of   (\ref {conical-kr-soliton})    is a critical point of  $\mu_{\omega_{g}, D}(\phi)$.    We need to study the
properness of  $\mu_{\omega_{g}, D}(\phi)$.

 Define
\beq\label{modified-L}
\mathcal L_{\beta, \tau}(u)=\beta \left(\mathcal L(u)-\int_{P} \langle\tau, \nabla u\rangle e^{\theta(x)} \, dx\right), ~\forall ~u\in\mathcal C_\infty.
\eeq
 By (\ref{vanishing-condition}), it is clear
\beq\label{futaki-conical}
\mathcal L_{\beta, \tau}(u)=0, ~\forall ~u=\sum  a_\alpha x_\alpha, ~(a_1,...,a_n)\in \mathbb R^n.
\eeq

\begin{lemma}\label{equal-two-energy}
Let  $\phi\in \mathcal H^{G_0}_{D}(\omega_g)$ and  $u$
be the Legendre  dual function of $\varphi_0+\phi$.
Then
\beqn\label{reduced-energy-conic}\mu_{\omega_g, D}(\phi)=\frac{(2\pi)^{n}}{V}
\mathcal F_{\beta,\tau}(u)+C,
\eeqn
where
\beq\label{reduced-conic}
\mathcal F_{\beta, \tau}(u) =
- \int_{P} \log\det(u_{ij})e^{\theta(x)}\,dx + \mathcal L_{\beta, \tau}(u),
\eeq
and $C$ is a constant.
\end{lemma}

\begin{proof}
It suffices to transform the latter two terms in (\ref{logkenergy}) under symplectic potentials.
Note that
\beqs
I_{\omega_g}(\phi)-J_{\omega_g}(\phi)
=-\frac{1}{V}\int_{M}\phi e^{\theta_{X}(\omega_\phi)} \frac{\omega_{\phi}^n}{n!}+\frac{1}{V}\int_{0}^{1}\int_{M}
\dot\phi_{s}e^{\theta_{X}(\omega_{\phi_{s}})}\frac{\omega_{\phi_{s}}^n}{n!}ds
\eeqs
Hence by similar expression as in Propositon \ref{reduced-energy} we have
\beqs
&&(1-\beta)(I_{\omega_g}(\phi)-J_{\omega_g}(\phi))
+\int_M \log\|s_D\|^2e^{\theta_X(\omega_\phi)}\frac{\omega^{n}_{\phi}}{n!}\\
&=&\frac{(2\pi)^n}{V}\left[-(1-\beta)\int_{\mathbb R^n}\varphi e^{X(\varphi)} \det D^2\varphi \, d\xi
+\beta\int_{\mathbb R^n}\langle\tau, \xi\rangle e^{X(\varphi)} \det D^2\varphi \, d\xi \right]\\
&&+\frac{(2\pi)^n}{V}(1-\beta)\int_{0}^{1}\int_{M}
\dot\phi_{s}e^{\theta_{X}(\omega_{\phi_{s}})}\frac{\omega_{\phi_{s}}^n}{n!}ds\\
&=&\frac{(2\pi)^n}{V}\left[-(1-\beta)\mathcal L(u)-\beta\int_P\langle \tau, \nabla u \rangle e^{\theta(x)}\,dx\right]
\eeqs
Combing with (\ref{reduced}), we obtain (\ref{reduced-conic}).
\end{proof}

To prove the properness of $\mathcal F_{\beta, \tau}(u)$, by (\ref{futaki-conical}),
it suffices to consider the  function space $\tilde{\mathcal C}_{\infty, \tau}$
which contains all the functions in $\mathcal C_{\infty}$ normalized at $\tau$.

\begin{proposition}\label{properness-conical}
If $\tau\in P$, then for any $0<\delta<1$, there exists $C_{\delta,\beta}>0$ such that
$$ \mathcal F_{\beta,\tau}(u)\geq \delta \int_P ue^{\theta(x)}\, dx-C_{\delta,\beta},
\ \ \forall u\in \tilde{\mathcal C}_{\infty, \tau}.$$
\end{proposition}

\begin{proof}
The proof is similar to Propositon \ref{f-proper}.
Note that by the definition of Legendre transform,  we have $(x-\tau)\cdot \nabla u-u\geq 0$ for $u\in \tilde{\mathcal C}_{\infty, \tau}$.
\footnote{For any $x_0\in P$, let $l_{x_0}(x)$ be the tangent function at $x_0$,
then $l_{x_0}(x)=\nabla u(x_0)(x-x_0)+u(x_0)$. Sine $u$ is normalized at $\tau$,
we always have $l_{x_0}(\tau)\leq 0$. It is easy to see that $\nabla u(x_0)(x_0-\tau)-u(x_0)=- l_{x_0}(\tau)\geq 0$.}
Then one can replace (\ref{inq2}) in Lemma \ref{stable} by
$$\mathcal L_{\beta,\tau}(x)=\beta\left(\int_P[(x-\tau)\cdot \nabla u-u]e^{\theta(x)}\, dx
+\int_P u e^{\theta(x)}\, dx\right)\geq \beta\int_P u e^{\theta(x)}\, dx$$ and check the arguments line by line
in Section 3.2.
\end{proof}

   Applying Lemma \ref{equal-two-energy} and Proposition  \ref{properness-conical},  we can give a new proof of following  Datar-Guo-Song-Wang Theorem  \cite{DGSW}.\footnote{ The statement here is a bit  different to  one in  \cite{DGSW}.}

\begin{theorem}\label{DGSW} Let $ X=\sum_{i=1}^n \theta_i\zeta_i\in\eta_T$.   Let $\overline\beta=\sup\{\beta|~ \beta l_i(\tau)<1, ~i=1,...,d\}$.  Suppose that $\tau\in P$. Then for any  $\beta\le \overline\beta$
  there exists  a unique toric invariant conical K\"ahler-Ricci soliton $\omega$  which solves (\ref{conical-kr-soliton})
 with
$D=\sum (1-\beta_i)D_i$  and $\beta_i=\beta l_i(\tau).$
Moreover, the conical angles of $\omega$  are $2\pi\beta_i$  along $D_i$.
\end{theorem}

 \begin{proof}  Since  higher order estimates depend on  $C^0$-estimate  for a family of K\"ahler potentials $\phi_t$ as in \cite{TZ1,  JMR} (also see \cite{DGSW}  for toric manifolds),  it  suffices to get the $C^0$-estimate
when we use   the continuity method to  solve (\ref{conical-kr-soliton}).   Since  $\mu_{\omega_g, D}(\phi)$ is monotonic  for $\phi_t$ as in the smooth  metrics case \cite{TZ1, CTZ},  the properness of  $\mu_{\omega_g, D}(\phi)$ implies that $I_{\omega_g}(\phi_t)$ is uniformly bonuded.
As a consequence, we get an upper bound of $\phi_t$.   The lower bound  can be also obtained by establishing a uniform lower bonud for the Green functions of $\omega_{\phi_t}$ as in \cite{Ma, CTZ} for smooth metrics. There is another way to get a
H\"older estimate for Legendre functions $u_t$ of $\phi_t$ by using an observation in  \cite{D3} for toric invariant metrics, if one knows the upper bound of $u_t$, which is equal to one of $\phi_t$.  Then $u_t$ is uniformly bounded,  and so  is $\phi_t$ ( also see \cite {SZ}).  In fact, the latter argument was   presented   by  Datar-Guo-Song-Wang   \cite{DGSW}  while  they got an upper bound of $u_t$ by studying  a class of  real Monge- Amp\'ere  equations  as done in \cite{WZ}.
The uniqueness of $\omega$  follows from the convexity of  $\mathcal F_{\beta, \tau}(u)$.

\end{proof}

In case that $ X$ is chosen as a soliton vector field on $M$, $\tau=0$ by (\ref{vanishing-condition}). Then in  Theorem \ref{DGSW},  
 $\beta_i=\beta\le1,~i=1,...,d, $ and $D=\sum (1-\beta)D_i$.  Thus Theorem \ref{conical-theorem} is a corollary of  Theorem \ref{DGSW}.

\end{document}